\documentclass{amsart}
\usepackage{amssymb,amsmath,amscd,xy,graphicx,textcomp}

\theoremstyle{plain}
\newtheorem{theorem}{Theorem}
\newtheorem{lemma}{Lemma}
\newtheorem{proposition}{Proposition}
\newtheorem{corollary}{Corollary}
\newtheorem{problem}{Problem}

\theoremstyle{definition}
\newtheorem{definition}{Definition}

\theoremstyle{remark}
\newtheorem{remark}{Remark}

\xyoption{arrow}

\xyoption{matrix}

\def\C{\mathbb{C}}
\def\R{\mathbb{R}}
\def\Z{\mathbb{Z}}

\def\tree{\mathcal{T}}

\title{The algebra of \bf {$SL_3(\C)$} conformal blocks}
\author{Christopher Manon}
\thanks{supported by NSF fellowship  DMS-0902710}

\begin{document}

\maketitle

\begin{abstract} 
We construct and study a family of toric degenerations of the Cox ring of the moduli of quasi-parabolic principal $SL_3(\C)$ bundles on a smooth, marked curve $(C, \vec{p}).$  Elements of this algebra have a well known interpretation as conformal blocks, from the Wess-Zumino-Witten model of conformal field theory.  For the genus $0, 1$ cases we find the level of conformal blocks necessary to generate the algebra.  In the genus $0$ case we also find bounds on the degrees of relations required to present the algebra. As a consequence we obtain a toric degeneration for the projective coordinate ring of an effective divisor on the moduli $\mathcal{M}_{C, \vec{p}}(SL_3(\C))$ of quasi-parabolic principal $SL_3(\C)$ bundles on $(C, \vec{p})$.  Along the way we recover positive polyhedral rules for counting conformal blocks. 
\end{abstract}

\smallskip

\tableofcontents

\smallskip

\section{Introduction}

Let $(C, \vec{p})$ be a smooth, complex, projective curve of genus $g$ with $n$ marked points.  In this paper we study the projective coordinate algebras of the moduli stack $\mathcal{M}_{C, \vec{p}}(SL_3(\C))$ of quasi-parabolic principal $SL_3(\C)-$bundles on $(C, \vec{p}).$   In particular, we construct toric degenerations of these algebras from polyhedra designed to compute the dimensions of the spaces $V_{C, \vec{p}}(\vec{\lambda}, L)$ of $SL_3(\C)$ conformal blocks.

For any trivalent graph $\Gamma$ with $n$ leaves and first Betti number (from now on referred to as the $genus$) equal to $g,$
we will define a polytope $CB_{\Gamma}^*,$  see Section \ref{state}.  Additionally we will define for any $n-$tuple $\vec{\lambda}$ of dominant $SL_3(\C)$ weights, and non-negative integer $L,$ a polytope $CB_{\Gamma}^*(\vec{\lambda}, L).$   We prove the following in Section \ref{gensec}.

\begin{theorem}\label{main}
Let $C$ be an $n-$marked curve of genus $g$, and let $\Gamma$ be a trivalent graph
of genus $g$ with $n$ leaves. The Cox ring (total coordinate ring) $V_{C, \vec{p}}(G)$ of $\mathcal{M}_{C, \vec{p}}(SL_3(\C))$ flatly degenerates to the semigroup algebra $\C[CB_{\Gamma}^*]$ associated to $CB_{\Gamma}^*.$

\begin{equation}
V_{C, \vec{p}}(SL_3(\C)) \Rightarrow \C[CB^*_{\Gamma}]\\
\end{equation}

\end{theorem}

\noindent
As a consequence, we obtain a toric degeneration of each projective coordinate ring coming from an effective line bundle on $\mathcal{M}_{C, \vec{p}}(SL_3(\C)).$ 

\begin{corollary}\label{mainweights}
Let $C, \vec{p}, \Gamma$ be as above, and let $\mathcal{L}(\vec{\lambda}, L)$ be the effective line bundle on $\mathcal{M}_{C, \vec{p}}(SL_3(\C))$ corresponding to the 
$n$-tuple $\vec{\lambda}$ of dominant $SL_3(\C)$ weights and the non-negative integer $L.$ The projective coordinate algebra $R_{C, \vec{p}}(\vec{\lambda}, L)$ corresponding to $\mathcal{L}(\vec{\lambda}, L)$ flatly degenerates to the semigroup algebra $\C[CB_{\Gamma}^*(\vec{\lambda}, L)].$

\begin{equation}
R_{C, \vec{p}}(\vec{\lambda}, L) \Rightarrow \C[CB_{\Gamma}^*(\vec{\lambda}, L)]\\
\end{equation}

\end{corollary}

We then obtain deeper structural properties of these algebras from an analysis of the polytope $CB_{\Gamma}^*,$  
in Section \ref{gensec} we prove the following. 

\begin{theorem}\label{gor}
The algebra $V_{C, \vec{p}}(SL_3(\C))$ is Gorenstein.
\end{theorem}

The Picard group of $\mathcal{M}_{C, \vec{p}}(G),$ calculated in \cite{LS}, is a product of $n$ copies of the character group of $B$ with a copy of $\Z.$

\begin{equation}
Pic(\mathcal{M}_{C, \vec{p}}(G)) = \mathcal{X}(B)^n \times \Z\\
\end{equation}

\noindent
The cone of effective line bundles is a subcone of  $\Delta^n \times \Z_{\geq 0} \subset \mathcal{X}(B)^n\times \Z,$ where $\Delta$ is the Weyl chamber of $G$ associated to $B.$   We let $\Delta_L$ denote the alcove of level $L$, consisting of the dominant weights $\lambda$ with $\lambda(H_{\theta}) \leq L$, where $H_{\theta}$ is the coroot corresponding to the longest root of $\mathfrak{g}.$  A result of  Kumar, Narasimhan, and Ramanathan \cite{KNR}, Beauville and Laszlo \cite{BL}, and also Faltings \cite{Fal}, establishes that when $\lambda_i \in \Delta_L$, the space of global sections $H^0(\mathcal{M}_{C, \vec{p}}(G), \mathcal{L}(\vec{\lambda}, L))$ can be identified with the vector space $V_{C, \vec{p}}(\vec{\lambda}, L)$ of conformal blocks on the curve $(C, \vec{p})$.    It follows that the Cox ring of $\mathcal{M}_{C, \vec{p}}(G)$ is the sum of all spaces of conformal blocks for the Lie algebra $\mathfrak{g}$ on the curve $(C, \vec{p}).$   In Sections \ref{g=0} and \ref{g=1} we give bounds on the levels of conformal blocks needed to present the algebra $V_{C, \vec{p}}(SL_3(\C))$ in the $g=0$ and $g=1$ cases. 

\begin{theorem}\label{g0}
For generic $(\mathbb{P}^1, \vec{p}) \in \mathcal{M}_{0, n},$
The algebra  $V_{\mathbb{P}^1, \vec{p}}(SL_3(\C))$ 
is generated by conformal blocks of level $1,$ and is presented by a homogenous ideal generated by forms of degree $2$ and $3.$
\end{theorem}

\noindent
In \cite{M4}, it was shown that $V_{\mathbb{P}, \vec{p}}(SL_3(\C))$ is a sub-algebra of a polynomial ring over the algebra of $SL_3(\C)-$invariant $n-$tensors, $R_n(sl_3(\C))\otimes\C[X]$ (see \cite{M4} for the definition of this object). With this in mind, we can combine this theorem with Proposition $3.5.1$ of \cite{U},  which characterizes conformal blocks in the genus $0$ case. 

\begin{corollary}
For generic $p_i,$ the algebra  $V_{\mathbb{P}^1, \vec{p}}(SL_3(\C))$ is generated by the tensors $f\otimes X \in [V(\lambda_1^*)\otimes \ldots \otimes V(\lambda_n^*)]^{sl_3(\C)}\otimes \C\{X\},$ with $\lambda_i $ a fundamental weight, which satisfy the following condition.  For any element $\phi_k = v_1 \otimes \ldots v_{k-1} \otimes b_{\lambda_k} \otimes v_{k+1} \ldots \otimes v_n \in V(\lambda_1)\otimes \ldots \otimes V(\lambda_n),$ where $b_{\lambda_k}$ is a highest weight vector, we have the following equalities. 

\begin{equation}
\sum_{\vec{m}_k, |\vec{m}_k| = \ell_k} \prod_{j \neq k} \binom{\ell_k}{m_j} (p_j - p_k)^{-m_j} <f | \prod_{j \neq k} [\rho_j(e_{\theta})]^{m_j}| \phi_k> = 0\\
\end{equation}
\noindent
where $\ell_k =  2 - \lambda_k(\theta),$  $\vec{m}_k = (m_1, \ldots, m_{k-1}, m_{k+1}, \ldots, m_n) \in \{0, 1, 2\}^{n-1}$, $|\vec{m}_k| = \sum_{j \neq k} m_j,$  and $e_{\theta} \in \mathfrak{sl_3(\C)}$ is the raising operator for the longest root $\theta = L_1 - L_3.$ The symbol $\rho_j$
indicates that the element is to act on the $j-$th component of the tensor $\phi_k.$
\end{corollary}

\noindent
Here we have specialized Proposition $3.5.1$ of \cite{U} to the case $\mathfrak{g} = sl_3(\C),$ with $L = 1.$  
We also combine Theorem \ref{g0} with an analysis of a special semigroup in Section \ref{g=1} to obtain a generating set in the genus $1$ case. 

\begin{theorem}\label{g1}
For generic $(C, \vec{p}) \in \mathcal{M}_{1, n},$
The algebra  $V_{C, \vec{p}}(SL_3(\C))$ 
is generated by conformal blocks of levels $1,2,$ and  $3.$ 
\end{theorem}

Vector spaces of conformal blocks $V_{C, \vec{p}}(\vec{\lambda}, L)$ for a simple Lie algebra $\mathfrak{g}$ originate from the Wess-Zumino-Witten model of conformal field theory.  Here $(C, \vec{p})$ is the marked curve, $L$ is a non-negative integer known as the $level$ and $\vec{\lambda}$ is an $n-$tuple of dominant weights for the algebra $\mathfrak{g}.$   See \cite{TUY}, \cite{U}, \cite{Be}, and \cite{L}, and the Bourbaki review by Sorger, \cite{S} for the construction of these spaces.  

The spaces $V_{0, 3}(\lambda, \eta, \mu, L)$ can be described purely with representation theoretic data for the algebra $\mathfrak{g},$ see \cite{TUY}, \cite{U}.  This has lead several authors to discover positive polyhedral counting rules for their dimensions.  By this we mean that the dimensions of  the spaces $V_{C, \vec{p}}(\vec{\lambda}, L)$ can be computed by counting the lattice points in a convex polytope.    Several rules for $SL_3(\C)$ have been found by Walton, Mathieu, Senechal, and Kirillov, \cite{KMSW}, (see Section \ref{cb3} below), along with partial results for $SL_n(\C)$, $n > 3.$   Although these rules are stated in  \cite{KMSW}, we were unable to find explicit proofs in the literature, part of the motivation for this paper is to write down such a proof.

Our method is to use the commutative algebra algebra structures attached to conformal blocks. In \cite{M4} we showed that there is a sheaf of multigraded algebras $V(G)$ on $\bar{\mathcal{M}}_{g, n}$ which extends the sheaf of Cox rings over the smooth curves. For a particular stable curve $(C, \vec{p})$, the multigraded components of the fiber of this sheaf are the spaces of conformal blocks $V_{C, \vec{p}}(\vec{\lambda}, L)$.

\begin{equation}
V_{C, \vec{p}}(G) = \bigoplus_{\vec{\lambda}, L} V_{C, \vec{p}}(\vec{\lambda}, L)\\
\end{equation}

\noindent 
For fixed $\vec{\lambda}, L$ we also obtain the flat subsheaf $\mathcal{R}(\vec{\lambda}, L) \subset \mathcal{V}(G),$ where the fiber over $(C, \vec{p})$ smooth is the  projective coordinate ring of $\mathcal{M}_{C, \vec{p}}(G)$ defined by $\mathcal{L}(\vec{\lambda}, L).$

\begin{equation}
R_{C, \vec{p}}(\vec{\lambda}, L) = \bigoplus_{N = 0}^{\infty} H^0(\mathcal{M}_{C, \vec{p}}(G), \mathcal{L}(\vec{\lambda}, L)^{\otimes N})\\
\end{equation}

 In \cite{M4} we give degenerations of $V_{C, \vec{p}}(G)$ to simpler, but not necessarily toric, algebras for all $G.$  We do this by first passing to nodal curves $(C,\vec{p}) \in \bar{\mathcal{M}}_{g, n},$ then using a filtration on the algebra of conformal blocks built from the aforementioned factorization rules of \cite{TUY}.  This effectively reduces the problem of constructing a toric degeneration of $V_{C,\vec{p}}(G)$ to constructing such a degeneration for the algebra of conformal blocks $V_{0, 3}(G)$ on a three pointed genus $0$ curve, see Section \ref{state}.  Any polyhedral counting formulas developed for the conformal blocks on this pointed curve are then prime suspects for a toric degeneration of $V_{0,3}(G).$   

In Sections \ref{valsec} and \ref{cb3} we show that polyhedral counting rules stated in \cite{KMSW} emerge from toric degenerations of the algebra $V_{0, 3}(SL_3(\C))$, and we describe the relevant polyhedra in Section \ref{state}.   This is accomplished by an analysis of this algebra as a sub-algebra of an algebra $R_3(sl_3(\C))\otimes\C[X].$ This is a polynomial ring in one variable over an algebra $R_3(sl_3(\C))$ which comes from classical invariant theory, this algebra, and more generally $R_n(sl_3(\C)),$ is studied in \cite{MZ}.  The algebra $R_n(sl_3(\C))$ is the Cox ring of the moduli of $n-$point configurations on the full flag variety of $SL_3(\C).$  In \cite{MZ}, the author and T. Zhou prove results on toric degenerations and presentation considerations for this "classical invariant theory case" to those presented here in the "quantum case."

 In the case $G = SL_2(\C)$ the fact that $SL_2(\C)$ tensor products are multiplicity free is enough to establish a large family of toric degenerations that are useful in describing the algebras $V_{C, \vec{p}}(SL_2(\C))$ and $R_{C, \vec{p}}(\vec{r}, L).$   In \cite{M4} degenerations of the algebra $V_{C, \vec{p}}(SL_2(\C))$ are constructed and shown to be isomorphic to members of a class of semigroup algebras $\C[P_{\Gamma}]$ originating from mathematical biology.  Here, $P_{\Gamma}$ is a semigroup which depends on a trivalent graph $\Gamma$ of genus $g$ with $n$ leaves.  In \cite{BW}, Buczynska and Wiesniewski prove that members of this class of semigroup algebras are generated in degree $1,$ with quadratic relations, when $\Gamma$ is a tree.  In a pair of papers, \cite{Bu}, \cite{BBKM}, Buczynska, Buczynski, Kubjas and Michalek showed that $\C[P_{\Gamma}]$ is in general generated in degree $g +1.$  In \cite{M4}, we found a particular graph whose semigroup algebra is generated in degrees $1, 2$, with relations generated by quadratics in the generators.   Abe analyzes the $n =0$ case in \cite{A}, and shows that $V_{C}(SL_2(\C))$ is generated in degree $1.$ The algebras $R_{C, \vec{p}}(\vec{r}, L)$ are studied in \cite{M7}, where they are shown to be generically generated in degree $1$ with quadratic relations when each entry of $\vec{r}$ and $L$ are even. 

Degenerations  have also been constructed in the $g = 0$ $SL_2(\C)$ case by Sturmfels and Xu, see \cite{StXu}. The algebra $V_{\mathbb{P}^1, \vec{p}}(SL_2(\C))$ was studied in \cite{StV}, where it was presented as a quotient of the projective coordinate ring of an even spinor variety $\mathcal{S}^+.$    The $g = 0, SL_2(\C)$ case is notable because much is known about the Mori cone of the associated coarse moduli.  Now that generators of the Cox ring are known in the $SL_3(\C)$ case (Theorem \ref{g0}), it would be interesting to study the birational geometry of the coarse moduli in this case as well. In general, we would like to have presentation data for the semigroup algebras $\C[CB_{\Gamma}^*],$ and $\C[CB_{\Gamma}^*(\vec{\lambda}, L)].$

\begin{problem}
Find a minimal generating set of the semigroup algebras $\C[CB_{\Gamma}^*],$ and $\C[CB_{\Gamma}^*(\vec{\lambda}, L)].$
\end{problem}
  
Having $3|V(\Gamma)|$ semigroups for each graph $\Gamma$ is a bit of an embarassment of riches.  The number of lattice points in these polytopes, indeed the multigraded Hilbert functions of $\C[CB_{\Gamma}^*]$ do not depend on the graph $\Gamma,$ or even the choice of a semigroup from our three possibilities at each trinode. It would be fruitful to describe how these polytopes are combinatorially related. 

\begin{problem}\label{PL}
For two trivalent graphs $\Gamma, \Gamma'$ with first Betti number $g$ and $n$ leaves, find a piecewise linear map $p_{\Gamma, \Gamma'}: CB_{\Gamma}^* \to CB_{\Gamma'}^*,$ which is a bijection on lattice points. 
\end{problem}

\noindent

An algebraist may wonder why we focus on toric degenerations if our concern is for a more refined understanding of the algebra $V_{C, \vec{p}}(G)$, as monomial ideal degenerations can be just as useful from the perspective of combinatorics.  First, monomial ideal degenerations can be produced from the binomial degenerations we consider here, and we believe that toric degenerations are more significant to related concerns in symplectic geometry in light of recent work of Kaveh and Harada, \cite{KH}.  In general, we hope that the techniques we employ here can be applied to other moduli of structures on surfaces.

\subsection{Acknowledgements}

We are grateful for many illuminating conversations with Weronika Buczynska, Noah Giansiracusa, Angela Gibney, Christian Haase, Nathan Ilten, Alan Knutson, Christian Korff,  Kaie Kubjas,  Anne Schilling, David Swinarski, and Zhengyuan Zhou.  We also thank the referees for suggestions and helpful comments.

\section{Description of the degenerations}\label{state}

In this section we describe the polytopes $CB_{\Gamma}^*.$  We will do this in steps, by first describing the construction from \cite{M4}.  First, we degenerate the algebra of conformal blocks $V_{C, \vec{p}}(SL_3(\C))$ to the invariants of a torus in a tensor product of many copies of $V_{0, 3}(SL_3(\C)).$  We begin by choosing a trivalent graph $\Gamma$ of genus $g$ with $n$ leaves labeled $\{1, \ldots, n\}.$

\begin{figure}[htbp]
\centering
\includegraphics[scale = 0.43]{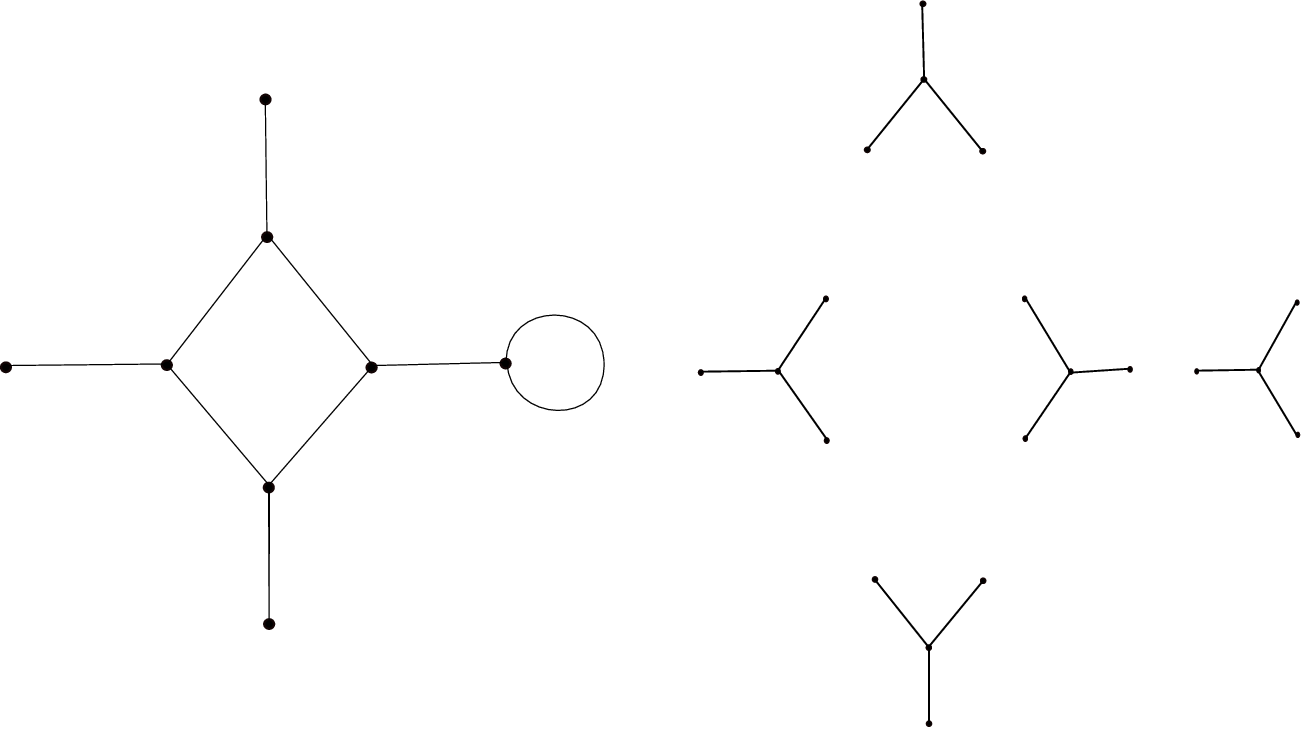}
\caption{A trivalent graph $\Gamma$ with associated forest $\hat{\Gamma}.$}
\label{fig:graph}
\end{figure}

We consider the forest $\hat{\Gamma}$ of trinodes obtained by splitting every edge of $\Gamma$ not connected to a leaf.  To each trinode $v \in V(\Gamma)$ we associate a copy of $V_{0, 3}(G)$ and make the tensor product $\bigotimes_{v\in V(\Gamma)} V_{0, 3}(G).$ This algebra carries the action of a large torus $T^{\Gamma},$ built by associating to each edge $e \in E(\hat{\Gamma})$ a product $T\times \C^*$ of a maximal torus $T \subset G$ with the non-zero complex numbers.  The torus $T^{\Gamma}$ acts on the component $\bigotimes_{v \in V(\Gamma)} V_{0, 3}(\lambda_v, \gamma_v, \eta_v, L_v) \subset \bigotimes_{v \in V(\Gamma)} V_{0, 3}(G)$ with the character obtained as the sum of differences $\sum_{v, v' \in V(\Gamma)} (\lambda_v^*, L_v) - (\lambda_{v'}, L_{v'})$ where $v$ and $v'$ are the end points of $e,$ and $\lambda_v, \lambda_{v'}$ are the weights assigned to the edges which map down to $e$ under the identification map $\hat{\Gamma} \to \Gamma.$

\begin{proposition}[M4]\label{propm4}
The algebra of invariants $[\bigotimes_{v\in V(\Gamma)} V_{0, 3}(G)]^{T^{\Gamma}}$ is the direct sum of the components   $\bigotimes_{v \in V(\Gamma)} V_{0, 3}(\lambda_v, \gamma_v, \eta_v, L_v)$ such that all the levels $L_v$ agree and $\lambda_v^* = \lambda_v$ for any two $v, v'$  related as above.
\end{proposition}

This algebra is multigraded by $networks$ of dominant $G-$weights, as depicted in Figure \ref{fig:network}.  The edge labeled "$\lambda$" is oriented to indicate that the "input" trinode receives the weight $\lambda^*$ and the "output" trinode receives the weight $\lambda.$

\begin{figure}[htbp]
\centering
\includegraphics[scale = 0.35]{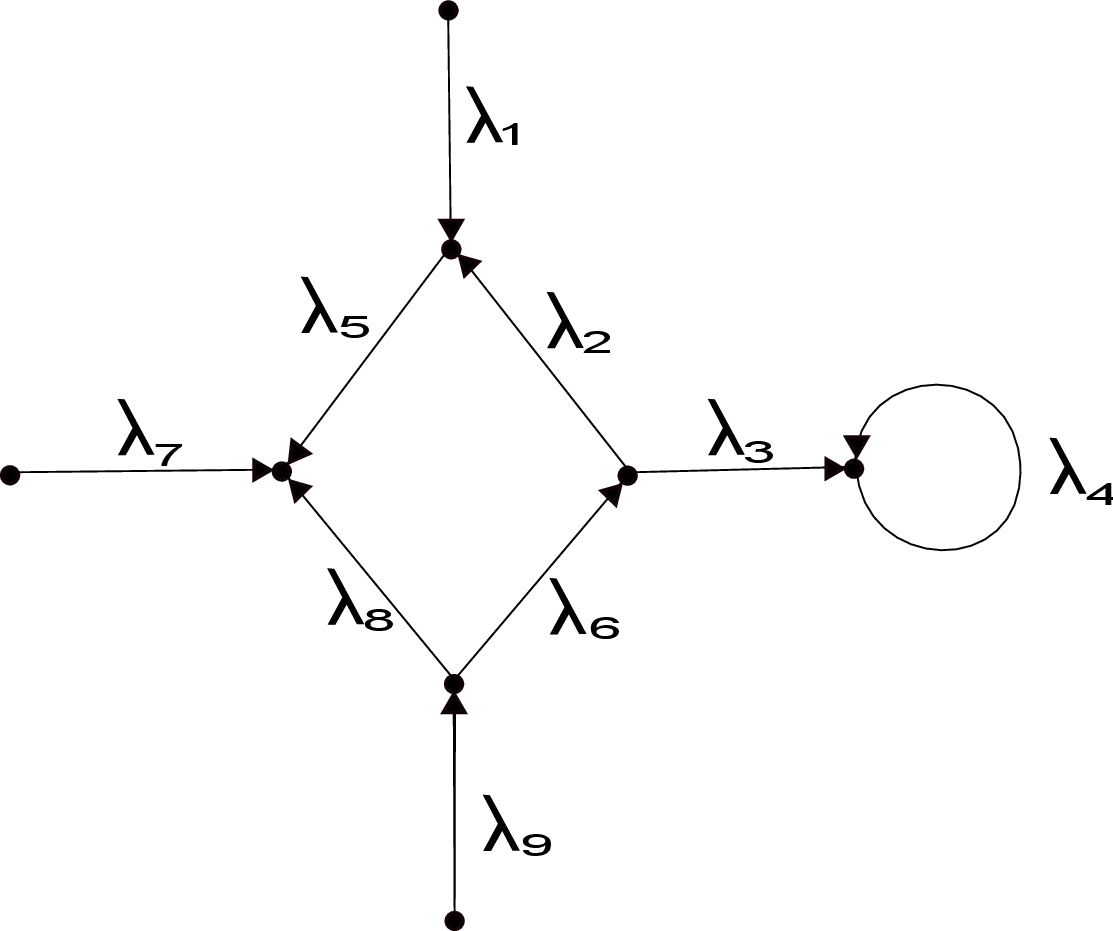}
\caption{}
\label{fig:network}
\end{figure}

\begin{theorem}[M4]\label{m4}
For every trivalent graph $\Gamma$ of genus $g$ with $n$ labeled leaves, there is a flat degeneration 

\begin{equation}
V_{C, \vec{p}}(G) \Rightarrow [\bigotimes_{v\in V(\Gamma)} V_{0, 3}(G)]^{T^{\Gamma}}\\
\end{equation}

\end{theorem}

The two vector spaces in this theorem are isomorphic by Tsuchiya, Ueno, Yamada's proof
of the factorization rules.  The content of the theorem is that the multiplication operation
in the algebra  $[\bigotimes_{v\in V(\Gamma)} V_{0, 3}(G)]^{T^{\Gamma}}$ is the "highest part" of the multiplication operation in $V_{C, \vec{p}}(G)$ with respect to a natural term order. This theorem reduces the problem to finding toric degenerations of $V_{0, 3}(SL_3(\C))$ which are $(T\times \C^*)^3-$invariant (essentially this is locating binomial points on a relevant multigraded Hilbert scheme, more on this point in Section \ref{cb3}).

We will see in Section \ref{cb3} that the algebra $V_{0, 3}(SL_3(\C))$ is generated by the components multigraded by the simple weights $(0, 1)$ and $(1, 0).$ These components are all multiplicity free, and so have a unique conformal block associated to them. The weights $(0, 1)$ and $(1, 0)$ are dual to each other, so our network graphical device above serves to give a presentation of $V_{0, 3}(SL_3(\C)),$ see Figure \ref{fig:trinodes}.

\begin{figure}[htbp]
\centering
\includegraphics[scale = 0.55]{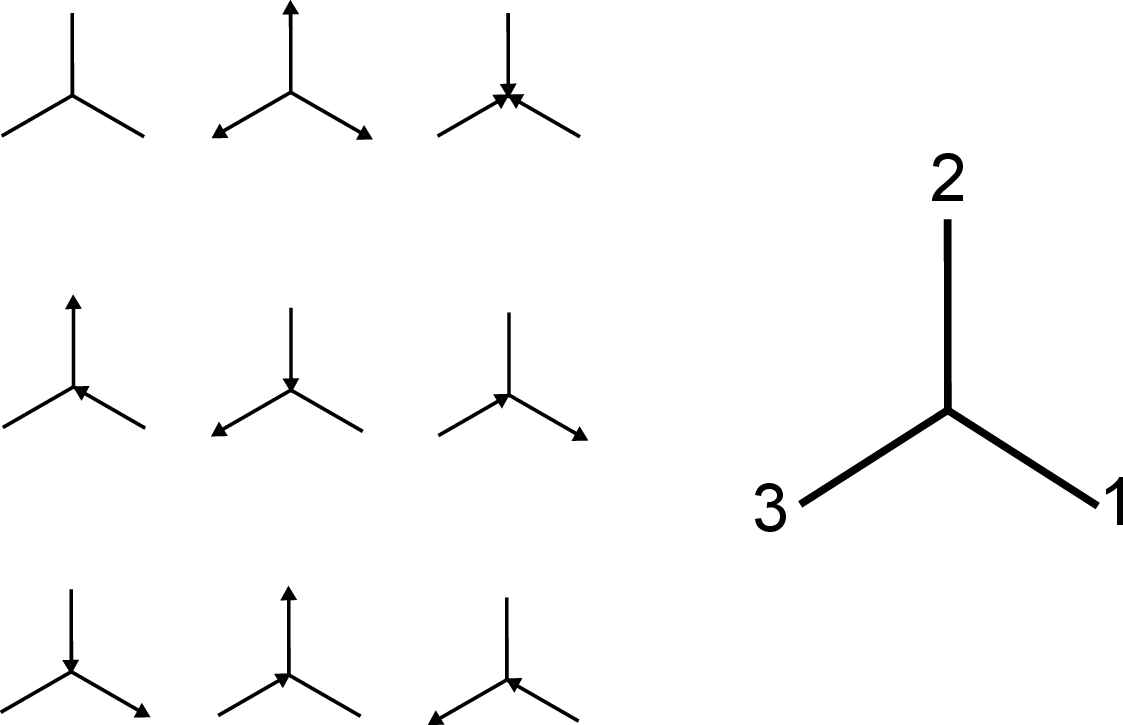}
\caption{}
\label{fig:trinodes}
\end{figure} 

We label these generators from left to right, top down, $X, S, T$, $ P_{12}, P_{23}, P_{31},$ $P_{21}, P_{32}, P_{13}.$  These come with a map $\partial$ to triples of dominant weights of $SL_3(\C).$

$$\partial(X) = [(0, 0), (0, 0), (0, 0)]$$

$$\partial(T) = [(0, 1), (0, 1), (0, 1)], \ \ \partial(S) = [(1, 0), (1, 0), (1, 0)],$$ 

$$\partial(P_{32}) = [(0, 0), (0, 1), (1, 0)], \ \ \partial(P_{12}) = [(1, 0), (0, 1), (0, 0)],$$ 

$$\partial(P_{13}) = [(1, 0), (0, 0), (0, 1)], \ \ \partial(P_{23}) = [(0, 0), (1, 0), (0, 1)],$$

$$\partial(P_{21}) = [(0, 1), (1, 0), (0, 0)], \ \ \partial(P_{31}) = [(0, 1), (0, 0), (1, 0)],$$

\noindent
We let $\partial_i$ denote the $i-$the component of the map $\partial.$
The three degenerations of $V_{0, 3}(SL_3(\C))$ are then given by the following three binomial relations. 

\begin{equation}
XST = P_{12}P_{23}P_{31}\\
\end{equation}

\begin{equation}
XST = - P_{21}P_{32}P_{13}\\
\end{equation}

\begin{equation}
 P_{12}P_{23}P_{31} = P_{21}P_{32}P_{13}\\
\end{equation}

Note that these relations define isomorphic semigroups, but non-isomorphic semigroups-with-map $\partial$ to $\R_{\geq 0}^2.$  This is important as the semigroups we obtain as degenerations of $V_{C, \vec{p}}(SL_3(\C))$ are constructed as fiber products of these "local" semigroups with respect to the map $\partial$ over the structure of the graph $\Gamma.$   As a result we actually obtain $3|V(\Gamma)|$ possible semigroup degenerations of $V_{C, \vec{p}}(SL_3(\C))$ per graph $\Gamma$, each one a potential tool of investigation.

We let $CB_3^*$ denote the semigroup defined
by the relation $P_{12}P_{23}P_{31} = P_{21}P_{32}P_{13}$ and $BZ_3^*$ denote
the semigroup defined by the relation $XST = P_{12}P_{23}P_{31}.$  We let $CB_3^*(\alpha, \beta, \gamma, L)$ and $BZ_3^*(\alpha, \beta, \gamma, L)$ denote
the $\partial-$fibers over $(\alpha, \beta, \gamma)$ in the $L-$th Minkowski sums $L \circ CB_3^*,$ $L \circ BZ_3^*,$ respectively. Both of these semigroups have appeared in the work of Walton, Senechal, Mathieu and Kirillov on counting $SL_3(\C)$ conformal blocks, \cite{KMSW}.  In particular, we recover their counting results.

\begin{corollary}
The dimension of the space $V_{0, 3}(\alpha, \beta, \gamma, L)$ is equal to the number of lattice points in the polytope $CB_3^*(\alpha, \beta, \gamma, L).$
\end{corollary}

 We will build a polytope $CB_{\Gamma}^*$, and let $CB_{\Gamma}^*(\vec{\lambda}, L)$ be the fiber over $\vec{\lambda}$ in $L \circ CB_{\Gamma}^*.$

\begin{corollary}
The dimension of the space $V_{C, \vec{p}}(\vec{\lambda}, L)$ is equal to the number of lattice points in the polytope $CB_{\Gamma}^*(\vec{\lambda}, L).$
\end{corollary}

 We construct the polytope $CB_{\Gamma}^*$ in a manner similar to $[\bigotimes_{v \in V(\Gamma)} V_{0, 3}(G)]^{T^{\Gamma}}.$ Once again, we consider the forest $\hat{\Gamma}$ and assign each trinode a copy of $CB_3^*.$ Next, we take the fiber product polytope by requiring that the $\partial$ values assigned at edges identified by $\hat{\Gamma} \to \Gamma$ be dual to each other. The lattice points of $CB_{\Gamma}^*$ can be visualized as oriented multigraphs as above.

\begin{figure}[htbp]
\centering
\includegraphics[scale = 0.65]{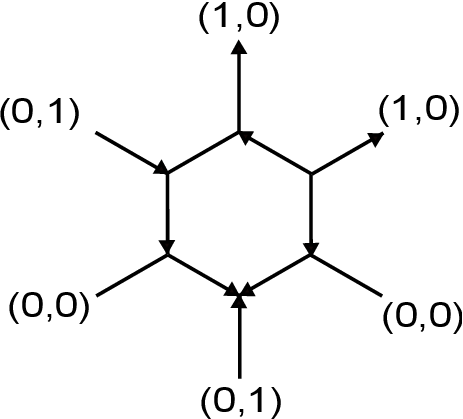}
\caption{}
\label{fig:trinet}
\end{figure}

\section{Commutative algebras associated to $sl_3(\C)$}\label{com}

We fix once and for all the choice of roots $L_1-L_2, L_2- L_3, L_1 - L_3$ for $sl_3(\C)$ and identify the set of dominant $sl_3(\C)$ weights with respect to this choice with $\Z_{\geq 0}^2.$ Recall that the fundamental irreducible representations of $sl_3(\C)$ are $V(1, 0) \cong \C^3$ and $V(0, 1) \cong \bigwedge^2(\C^3)$. We give these spaces the bases $x_1, x_2, x_3$ and $y_1, y_2, y_3,$ respectively, with $y_i = x_j \wedge x_k.$

We let $R$ denote the Cox ring of the full flag variety $SL_3(\C)/B.$  This is a $\Z_{\geq 0}^2$ graded, $SL_3(\C)$ algebra, with a full, multiplicity free decomposition into irreducible representations. 

\begin{equation}
R = \bigoplus_{\alpha \in \mathbb{Z}_{\geq 0}^2} V(\alpha)\\
\end{equation}

\noindent
The algebra  $R$ is a natural byproduct of the embedding of $SL_3(\C)/B$ into the variety $\mathbb{P}(\C^3)\times \mathbb{P}(\bigwedge^2(\C^3))$
given by sending a flag $\mathcal{F} = \{ \ell \subset w\}$ to its Pl\"ucker coordinates.  The pullbacks of the two generators of $Pic(\mathbb{P}(\C^3)\times \mathbb{P}(\bigwedge^2(\C^3)))$ generate the Picard group of $SL_3(\C)/B$ and give a multigraded surjection, 

\begin{equation}
Cox(\mathbb{P}(\C^3)\times \mathbb{P}(\bigwedge^2(\C^3))) = \C[x_1, x_2, x_3, y_1, y_2, y_3] \to R.\\
\end{equation}

\noindent
The kernel of this surjection is generated by the form $x_1y_1 + x_2y_3 + x_3y_3,$  which can be seen as the condition that the normal of $w \in \mathcal{F}$ vanish on the coordinates of $\ell \in \mathcal{F}.$  

We let $R_3(sl_3(\C))$ be the algebra of invariants $[R\otimes R\otimes R]^{sl_3(\C)}.$ As a vector space, this is the direct sum of all spaces of invariants in triple tensor products of irreducible $sl_3(\C)$ representations. 

\begin{equation}
R_3(sl_3(\C)) = \bigoplus_{\alpha, \beta, \gamma} (V(\alpha)\otimes V(\beta) \otimes V(\gamma))^{sl_3(\C)}\\
\end{equation}

\noindent
Our analysis of conformal blocks involves the algebra $R_3(sl_3(\C))$ and the following semigroup.

\begin{definition}
We let $BZ_3$ denote the semigroup of non-negative integer weightings of the diagram in Figure \ref{fig:BZ} such that the sums of pairs of weights opposite each other across the  hexagon are equal.

\begin{figure}[htbp]
\centering
\includegraphics[scale = 0.35]{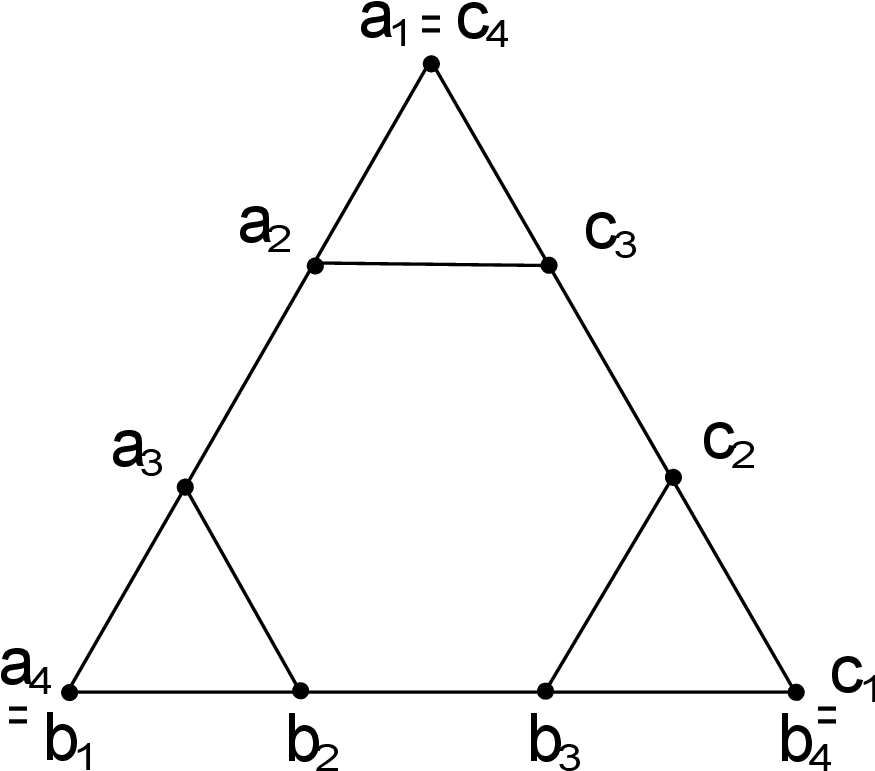}
\caption{}
\label{fig:BZ}
\end{figure} 

$$a_2+ a_3 = b_3 + c_2,$$

$$b_2+ b_3 = c_3 + a_2,$$

$$c_2+ c_3 = a_3 + b_2.$$
\end{definition}

We orient the triangle counter-clockwise.  For $a_1, a_2, a_3, a_4$ the consecutive weights along an edge, we obtain a map to the dominant weights of $sl_3(\C),$ $(a_1 + a_2, a_3 + a_4) \in \mathbb{Z}_{\geq 0}^2.$  For a triangle
$X \in BZ_3$, define $\partial(X) = (\alpha, \beta, \gamma) \in (\mathbb{Z}_{\geq 0}^2)^3$ to be the vector of the three dominant weights obtained this way from the three sides of $X$.  

\begin{theorem}[Berenstein, Zelevinksy, \cite{BZ1}]
The dimension $dim[(V(\alpha)\otimes V(\beta) \otimes V(\gamma))^{sl_3(\C)}]$ is equal to the number of BZ triangles in $\{X \in BZ_3 | \partial(X) = (\alpha, \beta, \gamma)\} = BZ_3(\alpha, \beta, \gamma).$ 
\end{theorem}

Triangles for a fixed  $(\alpha, \beta, \gamma)$ are all related to each other by adding or subtracting multiples of the triangle in Figure \ref{fig:BZminimal}.

\begin{figure}[htbp]
\centering
\includegraphics[scale = 0.35]{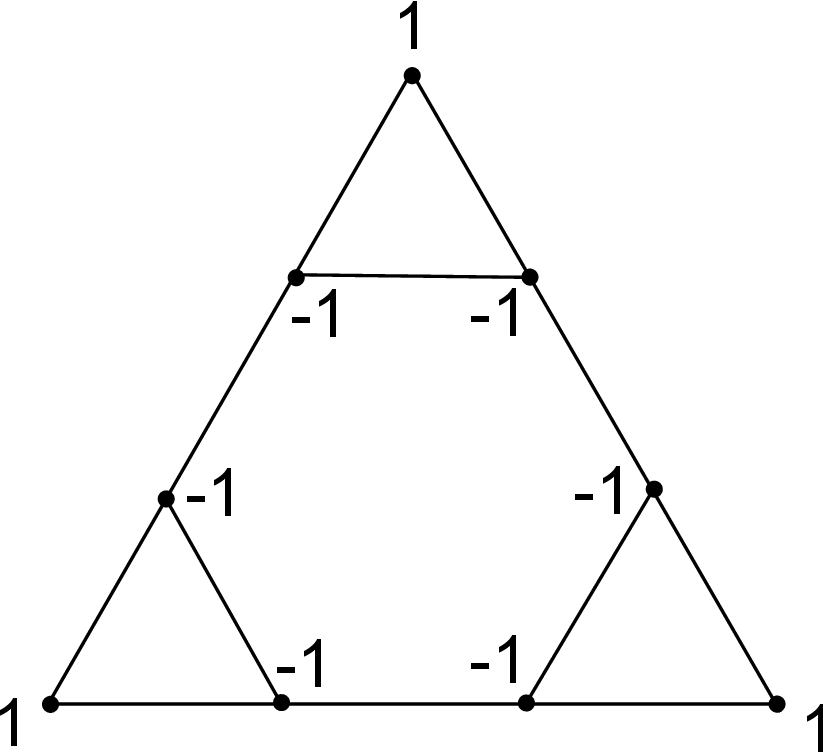}
\caption{}
\label{fig:BZminimal}
\end{figure} 

\begin{remark}
For two lattice cones $P \subset \R^n,$  $Q \subset \R^m,$ and a map $\pi: P \to Q,$
a collection of lattice points $\mathcal{M} \subset \Z^n$ is known as a Markov basis
for $\pi$ if the lattice points in any fiber $\pi^{-1}(b)$, $b \in Q$ can all be connected by
elements of $\mathcal{M}.$   The triangle above and its inverse therefore constitute a Markov basis for $\partial.$ 
\end{remark}

\begin{definition}
Define $Q_{min}(\alpha, \beta, \gamma) \in BZ_3(\alpha, \beta, \gamma)$ 
to be the unique triangle with some corner entry equal to $0.$
\end{definition}

It follows easily that all other members of the set $BZ_3(\alpha, \beta, \gamma)$ are obtained from $Q_{min}(\alpha, \beta, \gamma)$ by successive applications of the triangle in Figure \ref{fig:BZminimal}.  This element can be applied until some entry in the internal hexagon is $0$. 

\begin{remark}
It follows that the triples $\alpha, \beta, \gamma$ for which  $|BZ_3(\alpha, \beta, \gamma)| = 1$  are those which admit a triangle with a $0$ corner and a $0$ hexagon weight. 
\end{remark}

The semigroup $BZ_3$ is generated by the eight elements depicted in Figure \ref{fig:BZ2}, these elements have been represented by their honeycomb graphs. Honeycombs are a graphical device invented by Knutson, Tao, and Woodward \cite{KTW} to make the combinatorics of BZ triangles more transparent.  Each weight $a$ along the interior hexagon is replaced with an edge connected to the center of the triangle weighted by $a.$

\begin{figure}[htbp]
\centering
\includegraphics[scale = 0.3]{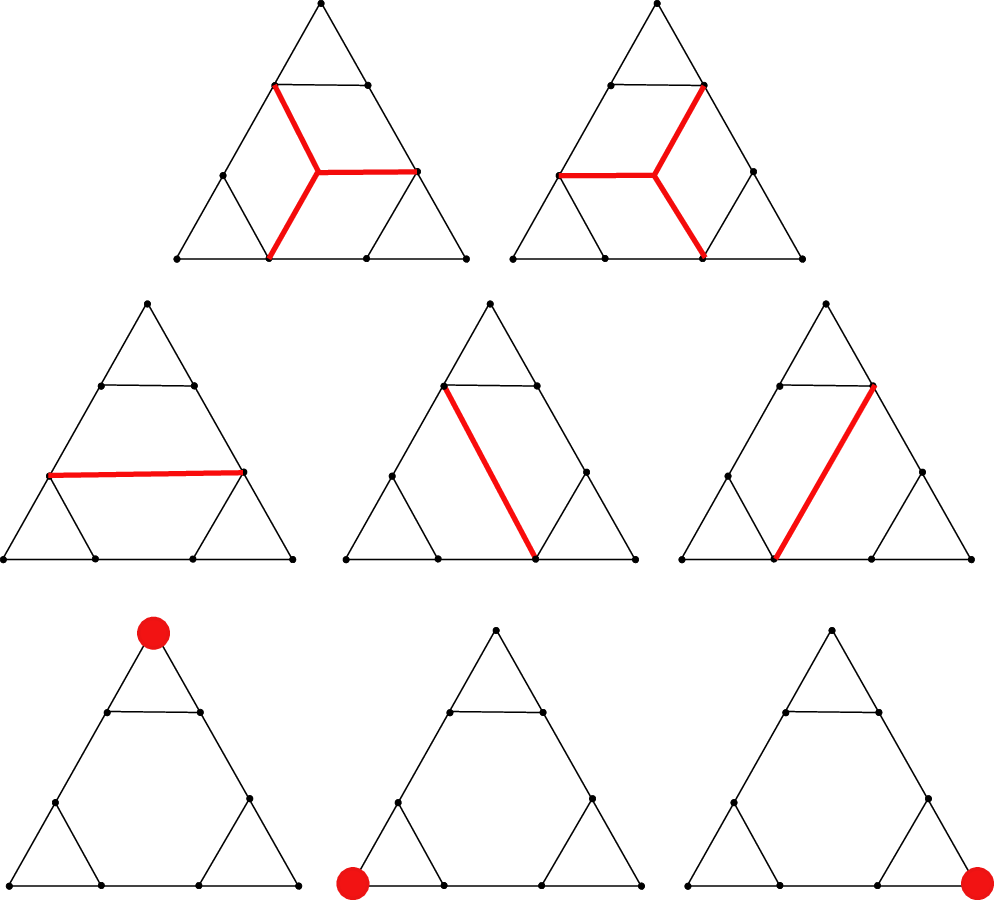}
\caption{}
\label{fig:BZ2}
\end{figure} 

From left to right, top down these are $\mathcal{S}, \mathcal{T}, \mathcal{P}_{12}, \mathcal{P}_{23}, \mathcal{P}_{31}, \mathcal{P}_{21}, \mathcal{P}_{32}, \mathcal{P}_{13}.$  These elements are subject to the familiar relation

\begin{equation}
\mathcal{S}\mathcal{T} =  \mathcal{P}_{12} \mathcal{P}_{23} \mathcal{P}_{31}.\\
\end{equation} 

We factor $Q_{min}(\alpha, \beta, \gamma)$ in the unique way which prefers $\mathcal{S}\mathcal{T}$ to $\mathcal{P}_{12} \mathcal{P}_{23} \mathcal{P}_{31}.$ 

\begin{equation}
Q_{min}(\alpha, \beta, \gamma) = \mathcal{S}^{s(min)}\mathcal{T}^{t(min)} \prod \mathcal{P}_{ij}^{p_{ij}(min)}\\
\end{equation}

\noindent
Letting $k = min\{t(min), s(min)\},$ the number of triangles in $BZ_3(\alpha, \beta, \gamma)$ is then equal to $k+1.$   We now relate these two algebraic objects by using $BZ_3$ to construct a presentation of $R_3(sl_3(\C)).$  The following appears in \cite{M5} and \cite{MZ}.

\begin{theorem}\label{br3}
The algebra $R_3(sl_3(\C))$ carries a $T^3-$invariant flat degeneration to the semigroup algebra $\C[BZ_3].$
\end{theorem}

Recall that we have fixed the bases of $V(1,0)$ and $V(0, 1)$ to be $x_1, x_2, x_3$ and $y_1, y_2, y_3$ respectively. We let $S$ be the unique invariant in $V(1, 0)^{\otimes 3},$ and $T$ be the unique invariant in $V(0, 1)^{\otimes 3}.$ These are the determinants of the matrices $X$ and $Y$ below, respectively.  Let $P_{ij}$ be the inner product of the $i-th$ column of $X$ with the $j-$th column of $Y$, this represents the invariant in $V(1,0) \otimes V(0, 1)$ where these representations are in the $i, j$ positions in the tensor product, respectively.  

$$
\begin{array}{c|ccc|cc|ccc|}
& x_1^1 & x_2^1 & x_3^1 & &  & y_1^1 & y_2^1 & y_3^1\\
X =& x_1^2 & x_2^2 & x_3^2 & &Y =  & y_1^2 & y_2^2 & y_3^2\\
& x_1^3 & x_2^3 & x_3^3 & &  & y_1^3 & y_2^3 & y_3^3 
\end{array}
$$

\noindent
These are the unique invariant forms in their corresponding tensor products, and therefore they must correspond to the eight generators of $BZ_3.$

Next we make use of the concept of a $subduction$ basis for a valuation on the commutative algebra $R_3(sl_3(\C)).$

\begin{definition}
Let $A$ be a commutative algebra over a field $\mathfrak{k},$ with a valuation $\mathfrak{v}: A \to \R\cup \{-\infty\}$ which restricts to the trivial valuation on $\mathfrak{k}.$  This defines a filtration on $A$, let $in_{\mathfrak{v}}(A)$ be the associated graded algebra. We say a set $X \subset A$ is a subduction basis for $(A, \mathfrak{v})$ if the initial forms $in_{\mathfrak{v}}(X)$ generate $in_{\mathfrak{v}}(A).$
\end{definition}

If the valuation $\mathfrak{v}$ is sufficiently nice, for example if the dimension of
the space of elements of $A_{\leq r}$ with value less than or equal to a fixed $r$ is finite, then the existence of finite subduction bases can be a powerful tool for studying $A$ (this happens for all algebras and valuations we consider).  In particular, $X$ generates $A$ as a $\mathfrak{k}$ algebra, and it is possible to construct a presentation $gr_{\mathfrak{v}}(A) = \mathfrak{k}[X]/ in_{\mathfrak{v}}(I),$ where $I$ is the ideal of presentation of $A$ by $X \subset A,$ see \cite{M5}, \cite{Ka}.  This is essentially the subduction algorithm and the "syzygy lifting" property of SAGBI bases, see \cite{St}, Chapter $11.$

  It follows that $S, T$ and the $P_{ij}$ are a subduction basis for $R_3(sl_3(\C))$ with respect to the filtration defined in \cite{M5}. This implies that $S, T, P_{ij}$ generate $R_3(sl_3(\C)),$ and that the binomial relation defining $BZ_3$ lifts to generate the presentation ideal of $R_3(sl_3(\C))$ by $S, T, P_{ij}$.  

\begin{equation}
R_3(sl_3(\C)) = \C[S, T, P_{12}, P_{13}, P_{23}, P_{21}, P_{31}, P_{32}]/< ST - P_{12}P_{23}P_{31} + P_{21}P_{32}P_{13}>\\ 
\end{equation}

 This relation can be constructed by writing the determinant of the product of two $3\times 3$  matrices, $X, Y$ as either $ST = det(X)det(Y^t)$ or $det(XY^t) =  P_{12}P_{23}P_{31} -  P_{21}P_{32}P_{13} + h,$   where $h$ involves the terms $P_{ii},$ $i \in \{1, 2, 3\}.$ Recall that all of these terms vanish as the defining relation of $R.$  Notice that the generator of this ideal has three binomial degenerations, two of which are isomorphic.  One gives a presentation of $\C[BZ_3],$ and the other defines a semigroup algebra for a semigroup $CB_3.$ 

\begin{equation}
\C[CB_3] =  \C[S, T, P_{12}, P_{13}, P_{23}, P_{21}, P_{31}, P_{32}]/< P_{12}P_{23}P_{31} - P_{21}P_{32}P_{13}>\\ 
\end{equation}

\noindent
Elements of the semigroup $CB_3$ can be represented by the $bird$ $feet$ diagrams depicted in Figure 
\ref{fig:birdfeet}. The
relation can be seen as the fact that the product of the elements in the second column equals
the product of the elements in the third column.

\begin{figure}[htbp]
\centering
\includegraphics[scale = 0.35]{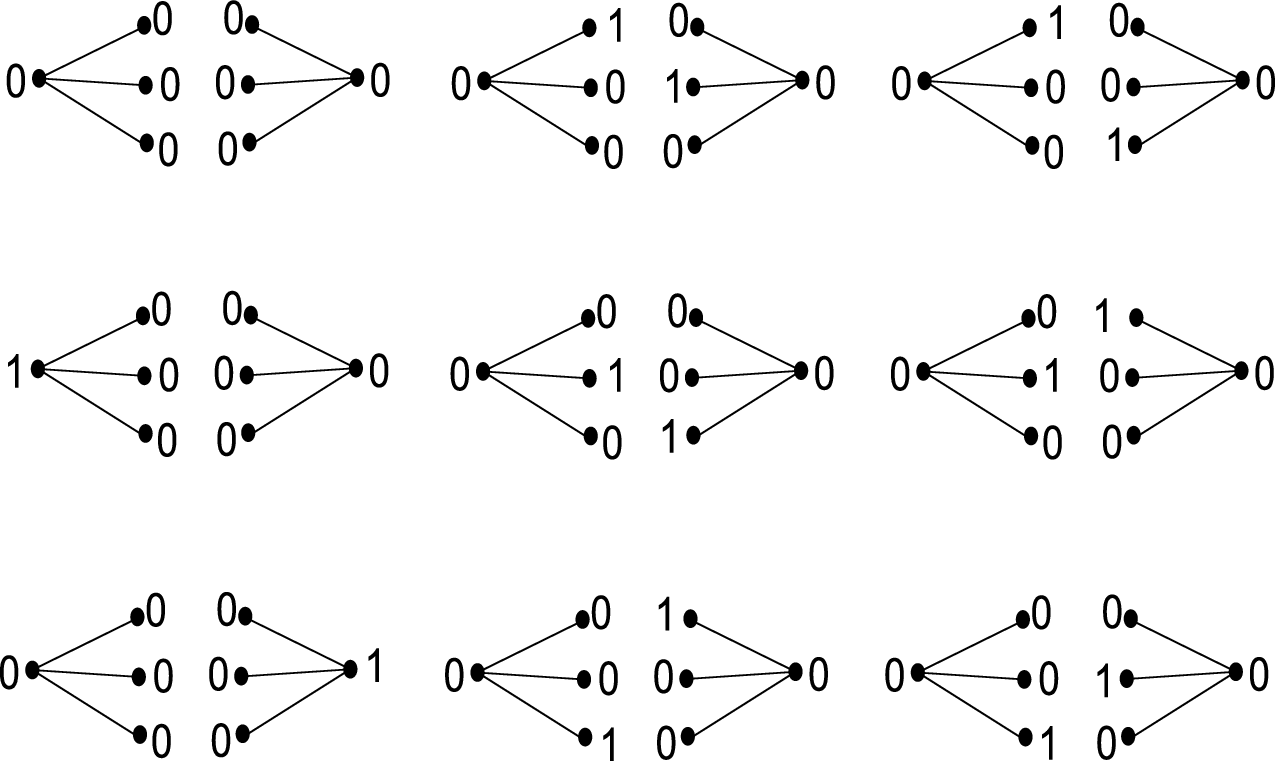}
\caption{}
\label{fig:birdfeet}
\end{figure}

Later we will relate $CB_3$ and $BZ_3$ to their "quantum" analogues
$CB_3^* \cong BZ_3^*,$ both obtained by introducing an extra generator $X.$

\section{Valuations on the algebra $R_3(sl_3(\C))$}\label{valsec}

In this section we define a valuation on the branching algebra
$R_3(sl_3(\C))$ which will help us construct and study the algebra
of $SL_3(\C)$ conformal blocks.  This function is built out of 
representation data for the subalgebra $sl_2(\C) \subset sl_3(\C)$
defined by the longest root $\theta = L_1 - L_3.$ First we decompose each irreducible representation of $sl_3(\C)$ into its $sl_2(\C)$ isotypical components with respect
to this root.

\begin{equation}
V(\alpha) = \bigoplus W_{\alpha, i}\\
\end{equation}

\begin{equation}
R^{\otimes 3} = \bigoplus W_{\alpha, i} \otimes W_{\beta, i} \otimes W_{\gamma, k}\\ 
\end{equation}

This gives a $branching$ $decomposition$ of the algebra $R^{\otimes 3}.$   
 For any map of groups $\phi:\mathfrak{h} \to \mathfrak{g},$ there is a filtration of the analogue of $R$ for the algebra $\mathfrak{g},$ $R_{\mathfrak{g}} = \bigoplus_{\lambda} V(\lambda)$ by its $\mathfrak{h}-$isotypical decomposition, see \cite{M5} and \cite{Gr}.   It follows that multiplication in the algebra $R^{\otimes 3}$ is lower-triangular with respect to the indices $i, j, k.$  This can be understood by considering the product of two such isotypical components.  Since the product map is $sl_2(\C)-$linear the image of the product must decompose as follows. 

\begin{equation}
[W_{\alpha_1, i_1}\otimes W_{\beta_1, j_1}\otimes W_{\gamma_1, k_1}] \times 
[W_{\alpha_2, i_2}\otimes W_{\beta_2, j_2}\otimes W_{\gamma_2, k_2}] \subset
\end{equation}
 $$\bigoplus_{i, j, k} [W_{\alpha_1 + \alpha_2, i}\otimes W_{\beta_1 + \beta_2, j} \otimes W_{\gamma_1 + \gamma_2, k}]$$

$$i_1 + i_2 \geq i, j_1 + j_2 \geq j, k_1 + k_2 \geq k$$

\noindent
Furthermore, a product $f\times g$ always has a component with indices $i = i_1 + i_2, j = j_1 + j_2, k = k_1 + k_2,$ \cite{M5}.  This implies that the function $v_{\theta},$ which assigns $\sum f_{\alpha, \beta, \gamma, i, j, k} = f \in R^{\otimes 3}$ the number $\frac{1}{2}max\{\ldots, i + j + k, \ldots\},$ is a valuation on $R^{\otimes 3}.$  Another construction of this type of filtration by dominant weight data appears in \cite{Gr}, $3.15.2$. From these observations we get the following proposition.

\begin{proposition}\label{val}
The function $v_{\theta}$ defines a valuation on the algebra $R_3(sl_3(\C)) \subset R^{\otimes 3}.$
\end{proposition}

\noindent

We note that the method of filtering an algebra by dominant weight data is used to prove Theorem \ref{m4} in \cite{M4}. As a consequence of the above proposition, it follows that $v_{\theta}(fg) = v_{\theta}(f) + v_{\theta}(g),$ $ie$ if $f$ has highest component of weight $(i_1 , j_1 , k_1),$ and $g$ has highest component of weight $(i_2 , j_2 , k_2)$ then $fg$ has highest component of weight $(i_1 + i_2 ,  j_1 + j_2 , k_1 + k_2).$  This is useful for studying conformal blocks because of the following (reformulated) theorem of Ueno. 

\begin{theorem}\label{ueno}[Ueno, \cite{U}]
The space of $sl_3(\C)$ conformal blocks of level $L$ and weights $\alpha^*, \beta^*, \gamma^*$ can be identified with $V_L(\alpha^*, \beta^*, \gamma^*) = \{f \in (V(\alpha) \otimes V(\beta) \otimes V(\gamma))^{sl_3(\C)} | v_{\theta}(f) \leq L\}$
\end{theorem}

\begin{proof}
We translate from the result of Ueno, consider $V(\alpha) \otimes V(\beta) \otimes V(\gamma)$ as linear functions on the representation $V(\alpha^*) \otimes V(\beta^*) \otimes V(\gamma^*),$ and consider the isotypical decomposition of this representation,

\begin{equation}
V(\alpha^*) \otimes V(\beta^*) \otimes V(\gamma^*) = \bigoplus W_i \otimes W_j \otimes W_j,\\
\end{equation}

\noindent
then by \cite{U}, Corollary $3.5.2$ we have

\begin{equation}
V_L(\alpha^*, \beta^*, \gamma^*) \cong \{f |  i + j + k > 2L \Rightarrow f|_{W_i \otimes W_j \otimes W_k} = 0\}\\ 
\end{equation}

Consider the $sl_2(\C)$ isotypical decomposition of $[V(\alpha^*) \otimes V(\beta^*) \otimes V(\gamma^*)]^*.$

\begin{equation}
[V(\alpha^*) \otimes V(\beta^*) \otimes V(\gamma^*)]^* = [\bigoplus W_i \otimes W_j \otimes W_k ]^* = \bigoplus W_i \otimes W_j \otimes W_k\\
\end{equation}

This follows from the self-duality of $sl_2(\C)$ representations.  As a consequence, the dual functions of the $W_i\otimes W_j \otimes W_k$ component of $V(\alpha^*) \otimes V(\beta^*) \otimes V(\gamma^*)$ are given by the $W_i\otimes W_j \otimes W_k$ component of $V(\alpha) \otimes V(\beta) \otimes V(\gamma).$ A function $f$ vanishes on $all$ $W_i \otimes W_j \otimes W_k,$ $i + j + k > 2L$ if and only if it is in the span of the components with $i + j + k \leq 2L.$
\end{proof}

\begin{lemma}
The invariants $P_{ij},$ $S$, $T$ each have $v_{\theta}$ value $1.$
\end{lemma}

\begin{proof}
We compute the decompositions of $V(0,1)$ and $V(1, 0)$ as $\theta(sl_2(\C)) \subset sl_3(\C)$ representations.  These are $V(1, 0) = V(1) \oplus V(0)$
with $V(0)$ the span of $x_2,$ and $V(0, 1) = V(1) \oplus V(0)$ with $V(0)$ the
span of $y_2.$   Each monomial in the expressions for $S, T,$ and $P_{ij}$ is then in a component with at most two copies of $V(1)$ in the tensor product. 
\end{proof}

\begin{proposition}
For $\alpha, \beta, \gamma$ fixed, the invariants in $V(\alpha)\otimes V(\beta) \otimes V(\gamma)$ take on distinct, consecutive $v_{\theta}$ values. 
\end{proposition}

To prove this proposition we use the BZ triangles to keep track of the dimension of 
the space of invariants $(V(\alpha)\otimes V(\beta) \otimes V(\gamma))^{sl_3(\C)}.$
To the minimum triangle $Q_{min}(\alpha, \beta, \gamma)$ above we associate the following monomial in $R_3(sl_3(\C)).$

\begin{equation}
M_{min} = S^{s(min)}T^{t(min)} \prod P_{ij}^{p_{ij}(min)}\\
\end{equation}

This monomial has weight $(\alpha, \beta, \gamma)$ by construction, and therefore
defines an invariant in  $(V(\alpha)\otimes V(\beta) \otimes V(\gamma))^{sl_3(\C)}.$
Furthermore, we can compute the $v_{\theta}$-value of this monomial, 

\begin{equation}
v_{\theta}(M_{min}) = s(min) + t(min) + \sum p_{ij}(min)\\ 
\end{equation}

Recall $k = min\{s(min), t(min)\},$ we produce $k$ more monomials in $(V(\alpha)\otimes V(\beta) \otimes V(\gamma))^{sl_3(\C)},$ one associated to each $1 \leq \ell \leq k$ as follows. 

\begin{equation}
M_{\ell} = (ST)^{-\ell}(P_{12}P_{23}P_{31})^{\ell} M_{min}\\ 
\end{equation}

\noindent
Each application of the Laurant monomial $(ST)^{-1}(P_{12}P_{23}P_{31})^{1}$ has the effect of raising the value of $v_{\theta}$ by $1$ while keeping the monomial in $(V(\alpha)\otimes V(\beta) \otimes V(\gamma))^{sl_3(\C)}.$  In this way we get $k+1$ monomials

\begin{equation}
M_{min} = M_0, M_1, \ldots, M_k\\
\end{equation}

\noindent
with $k+1$ distinct $v_{\theta}$ values $v_{\theta}(M_{min}), v_{\theta}(M_{min}) +1, \ldots, v_{\theta}(M_{min}) + k,$ where

\begin{equation}
v_{\theta}(M_{\ell}) = v_{\theta}(M_{min}) + \ell\\
\end{equation}

It follows from general properties of valuations that these monomials must be linearly independent, and we have already established that the invariant space is of dimension exactly $k+1,$ this proves the proposition.  We can extend this proposition to the following theorem.

\begin{theorem}\label{agrade}
The associated graded algebra of $R_3(sl_3(\C))$ with respect to the valuation $v_{\theta}$ is isomorphic to 
\begin{equation}
\C[S, T, P_{12}, P_{13}, P_{23},P_{21}, P_{31}, P_{32}]/ <P_{12}P_{23}P_{31} - P_{21}P_{32}P_{13}>.\\
\end{equation}
\end{theorem}

\begin{proof}
The proposition above shows that the initial forms of the monomial generators $S, T, P_{ij}$ generate the associated graded algebra.  This implies that they are a subduction basis for $R_3(sl_3(\C))$ with respect to the valuation $v_{\theta}.$  From this it follows that

\begin{equation}
gr_{v_{\theta}}(R_3(sl_3(\C))) \cong \C[X]/in_{v_{\theta}}(I)\\
\end{equation}

\noindent
For any set of generators $X \subset R_3(sl_3(\C))$ containing our subduction basis, and
$I$ the associated ideal of presentation.  In this case we have

\begin{equation}
I = <P_{12}P_{23}P_{31} - P_{21}P_{32}P_{13} + ST>,\\
\end{equation}

\noindent
and the initial form of the principal generator of this ideal is $P_{12}P_{23}P_{31} - P_{21}P_{32}P_{13}.$ 
\end{proof}

It also follows that the monomials $M_0, M_1, \ldots, M_{\ell}$ form a basis of the space
$V_{v_{\theta}(M_{min}) + \ell}(\alpha^*, \beta^*, \gamma^*) \subset (V(\alpha) \otimes V(\beta) \otimes V(\gamma))^{sl_3(\C)}.$

\section{The algebra of conformal blocks $V_{0, 3}(SL_3(\C))$}\label{cb3}

In this section we translate our results on $R_3(sl_3(\C))$ directly over to the algebra $V_{0, 3}(SL_3(\C)).$ We also employ tropical geometry and the algebraic geometry of a simple Hilbert scheme to depict how commutative algebras which "count" conformal blocks are related.   We let $\mathbb{V}_{0, 3}$ denote the multigraded Hilbert function of $V_{C, \vec{p}}(SL_3(\C)).$ In \cite{M4}, the algebra of conformal blocks $V_{0, 3}(SL_3(\C))$ was shown to be isomorphic to the following subalgebra of $R_3(sl_3(\C))\otimes \C[X],$ with $X$ a variable.

\begin{equation}
V_{0, 3}(SL_3(\C)) = \bigoplus_{\alpha, \beta, \gamma, k \leq L} V_k(\alpha, \beta, \gamma) \otimes \C X^L\\ 
\end{equation}

From the results of the last section we get the following theorem.

\begin{theorem}
The algebra $V_{0, 3}(SL_3(\C))$ is generated by the elements $S \otimes X, T \otimes X, P_{ij}\otimes X,$ and $1 \otimes X \subset R_3(sl_3(\C)) \otimes \C[X],$ subject to the relation

\begin{equation}
[S\otimes X][T \otimes X][1 \otimes X] - [P_{12}\otimes X][P_{23}\otimes X][P_{31}\otimes X] + [P_{21}\otimes X][P_{32}\otimes X][P_{13}\otimes X]\\
\end{equation} 
\end{theorem}

Furthermore, we can extend the valuation $v_{\theta}$ to a valuation on $V_{0, 3}(SL_3(\C)).$ The associated graded algebra of this valuation has the following presentation. 

\begin{equation}
gr_{v_{\theta}}(V_{0, 3}(SL_3(\C)))\\
\end{equation}
$$ = \C[S, T, X, P_{12}, P_{23}, P_{31}, P_{21}, P_{32}, P_{13}]/<P_{12}P_{23}P_{31} -  P_{21}P_{32}P_{13}>$$ 

\noindent
This can be seen to be isomorphic to the semigroup algebra of $CB_3^*.$ 

The three toric degenerations of the variety defined by the equation $ P_{12}P_{23}P_{31} - P_{21}P_{32}P_{13} + STX$  are indexed by the maximal cones
of the associated $tropical$ $variety$ $\mathbb{T}( P_{12}P_{23}P_{31} - P_{21}P_{32}P_{13} + STX) \subset \R^9,$ for the rudiments of tropical geometry see \cite{JSST}. 
The three cones of this tropical variety are defined by the inequalities $ S + T + X = P_{12} + P_{23} + P_{31} \geq P_{21} + P_{32} + P_{13},$ $ S + T + X \leq P_{12} + P_{23} + P_{31} = P_{21} + P_{32} + P_{13},$  and $ S + T + X = P_{21} + P_{32} + P_{13} \geq  P_{12} + P_{23} + P_{31},$ and are glued together along $ S + T + X = P_{12} + P_{23} + P_{31} = P_{21} + P_{32} + P_{13}.$  After quotienting away the lineality space defined by  $S + T + X = P_{12} + P_{23} + P_{31} = P_{21} + P_{32} + P_{13}$ we obtained a trinode, each branch of which represents a toric degeneration of $V_{0, 3}(SL_3(\C))$, see Figure \ref{fig:Hilbert}.

Another meaningful geometric structure which indexes algebras which can be used to count conformal blocks is the Hilbert scheme $\mathcal{H}_9^{\mathbb{V}_{0,3}}$ of ideals in the polynomial ring on the variables in Figure \ref{fig:trinodes}, with multigraded Hilbert function $\mathbb{V}_{0,3}.$ The following can be found in \cite{KMSW}, and can also be derived direction from our description of $V_{0, 3}(SL_3(\C))$. 

\begin{equation}
\mathbb{V}_{0,3}(\lambda, \mu, \eta, L) =  L -  max\{\lambda_1 + \lambda_2, \mu_1 + \mu_2, \eta_1 + \eta_2, L_1, L_2\} +1\\
\end{equation}

\noindent
where

\begin{equation}
L_1 =  \frac{1}{3}(2(\lambda_1 + \mu_1 + \eta_1) + \lambda_2 + \mu_2 + \eta_2) - min\{\lambda_1, \mu_1, \eta_1\}\\
\end{equation}

\noindent
and

\begin{equation}
L_2 = \frac{1}{3}(2(\lambda_2 + \mu_2 + \eta_2) + \lambda_1 + \mu_1 + \eta_1) - min\{\lambda_2, \mu_2, \eta_2\}.
\end{equation}

This expression holds when it is a non-negative integer, otherwise it is set to be $0.$  In order for an ideal $I \subset \C[S, T, X, P_{12}, P_{23}, P_{31}, P_{21}, P_{32}, P_{13}]$ to yield an algebra with this same multigraded Hilbert function, it must be principally generated by a polynomial  involving only the monomials $STX,  P_{12}P_{23}P_{31}, P_{21}P_{32}P_{13}.$  For this reason, 
$\mathcal{H}_9^{\mathbb{V}_{0,3}} = \mathbb{P}^2,$  and any semigroup algebra generated by $ S, T, X,$ $P_{12}, P_{23}, P_{31},$ $P_{21}, P_{32}, P_{13} $ which counts $SL_3(\C)$ conformal blocks must lie on one of the copies of $\mathbb{P}^1$ defined by linear combinations of two of these monomials, see  Figure \ref{fig:Hilbert}.

\begin{figure}[htbp]
\centering
\includegraphics[scale = 0.55]{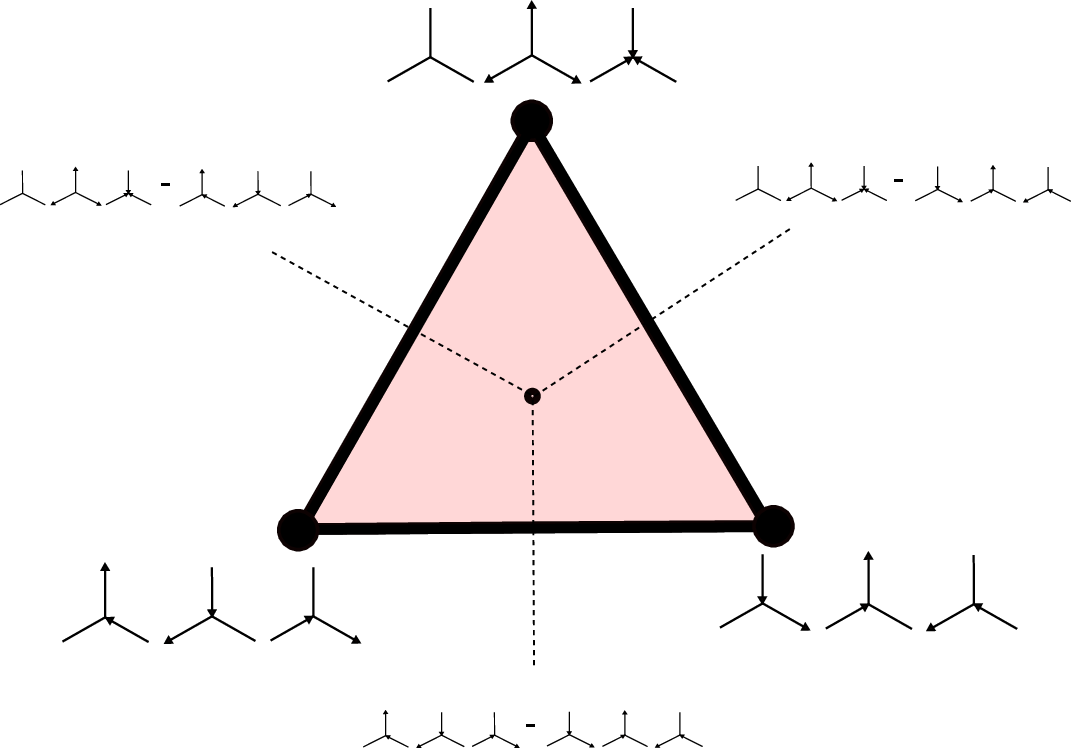}
\caption{The Hilbert scheme $\mathcal{H}_9^{\mathbb{V}_{0,3}}$, with dual tropical variety $\mathbb{T}( P_{12}P_{23}P_{31} - P_{21}P_{32}P_{13} + STX )$.}
\label{fig:Hilbert}
\end{figure}

\section{The algebra of conformal blocks $V_{C, \vec{p}}(SL_3(\C))$}\label{gensec}

In this section we use results from \cite{M4} and the previous section 
to obtain toric degenerations of the algebra of $SL_3(\C)$ conformal blocks 
over a general marked stable curve $(C, \vec{p}).$  We also use these degenerations
to show that $V_{C, \vec{p}}(SL_3(\C))$ is a Gorenstein algebra.   

\begin{definition}
Let $\mathfrak{k}$ be a field, then a $\Z$-graded $\mathfrak{k}$-algebra $R$ is said to be $Gorenstein$ if the Matlis Dual

\begin{equation}
H_{m}^{dim(R)}(R)^* = \underline{Hom}_{\mathfrak{k}}(H_{m}^{dim(R)}(R), \mathfrak{k}),\\
\end{equation}

\noindent
is isomorphic to grade-shifted copy $R(-a)$ of $R$. Here $m$ is the 
maximal ideal generated by elements in $R$  of positive degree, and $\underline{Hom}_{\mathfrak{k}}(-,-)$ is the functor of graded $\mathfrak{k}$-morphisms.  The number $a$ is called the $a$-invariant of the graded Gorenstein algebra $R$, see \cite{BH}.
\end{definition}

\noindent
For an affine semigroup algebra this property is expressed as follows. 

\begin{proposition}
Let $S_P$ be the semigroup given by the lattice
points in a polyhedral cone $P$.  Then the algebra
$\C[S_P]$ is Gorenstein if and only if 
there is a lattice point $\omega \in int(P)$ with $int(P) \cap S_P = \omega + S_P.$
Furthermore, in the presence of a grading, 
we have $a(\C[S_P]) = deg(\omega).$ 
\end{proposition}

This is Corollary 6.3.8 in \cite{BH}. By a theorem of Stanley (see \cite{BH}, Corollary $4.4.6$), the Gorenstein property on a graded domain depends only on details of its Hilbert function.  Since this is unchanged by flat degeneration, we reduce establishing this property for $V_{C, \vec{p}}(SL_3(\C))$ to the case $\C[BZ_{\Gamma}^*],$ given Theorem \ref{main}. 

To prove Theorem \ref{main}, we can combine the degeneration described in Theorem \ref{m4} with the one defined on each component $V_{0, 3}(SL_3(\C))$ in the previous section, it remains only to translate operations on semigroups to operations on their corresponding semigroup algebras.   We have $\C[(CB_3^*)^{v \in V(\Gamma)}] = \bigotimes_{v \in V(\Gamma)} \C[CB_3^*]$.  This algebra carries an action of the torus $T^{\Gamma},$ which appeared in the statement of Theorem \ref{m4}.  The algebra $\C[CB_{\Gamma}^*]$ can be constructed as the invariant subalgebra of $\otimes_{v \in V(\Gamma)} \C[CB_3^*]$ with respect to the torus $T^{\Gamma}.$ 

\begin{equation}
\C[CB_{\Gamma}^*] = (\bigotimes_{v \in V(\Gamma)} \C[CB_3^*])^{T^{\Gamma}}\\
\end{equation}

All degenerations we have constructed are $T^{\Gamma}-$invariant, so this proves Theorem \ref{main}.

\begin{theorem}
The algebra $V_{C, \vec{p}}(SL_3(\C))$ has the structure of a graded
Gorenstein algebra, with $a-$invariant equal to $6.$
\end{theorem}

\begin{proof}
This is a consequence of the following Lemmas. 
\end{proof}

\begin{lemma}\label{intersect}
If $P$ is polytope with $\C[P]$ a Gorenstein semigroup algebra with $\omega \in a\circ P$, and if $Q = P \cap \mathcal{H}$ 
is polytope obtained from $P$ by intersecting with a subspace $\mathcal{H}$ in a way such that $\omega \in a\circ Q$ then $\C[Q]$ is a Gorenstein semigroup algebra with $*-$canonical generator $\omega.$
\end{lemma}

\begin{proof}
The proper facets of $Q$ are all obtained as intersections of proper facets of $P$ with $\mathcal{H},$ 
therefore a point in the relative interior of $P$ which is also in $\mathcal{H}$ is also in the relative interior
of $Q.$  Since all interior lattice points in the Minkowski sum $k\circ Q$ are of the form $kX$ for $X$ in the relative
interior of $Q,$ this shows that any lattice point $X$ in the interior must be divisible by $\omega.$
\end{proof}

\begin{lemma}
$BZ_3^*$ is Gorenstein with $\omega_3 = P_{12}P_{23}P_{31}XST.$
\end{lemma}

\begin{proof}
This follows from Lemma \ref{intersect} above.  The polytope $BZ_3^*$ can be obtained from the simplex
in $\R^9$ given as the convex hull of the origin and the points $(0, \ldots, 1, \ldots, 0).$  One takes the $3$rd
 Minkowski sum of this simplex and intersects it with the hyperplanes defined by the hexagon relations.    
The generator of the $*-$canonical module of this simplex is the point $(1, 1, 1, 1, 1, 1, 1, 1, 1),$
which agrees with $\omega_3 \in BZ_3^*.$
\end{proof}

\begin{lemma}
If $P_1$ and $P_2$ are normal polytopes which give Gorenstein
semigroup algebras with $*-$canonical generators $\omega_1$ and
$\omega_2$ in the same degree, then $\C[P_1 \times P_2]$ is Gorenstein with
$*-$canonical generator $\omega_1 \times \omega_2$.
\end{lemma}

\begin{proof}
This follows from the fact that any point $(X, Y)$ in the interior of $k\circ (P_1 \times P_2)$
must have $X$ and $Y$ in the interiors of $P_1$ and $P_2$, respectively. 
\end{proof}

Since all $BZ_{\Gamma}^*$ are obtained by intersecting a product $(BZ_3^*)^{V(\Gamma)}$ with a subspace,
this proves the theorem. The generator $\omega_{\Gamma}$ of the $*-$canonical module is the multigraded component
$V_{C, \vec{p}}((2, 2), \ldots, (2,2), 6),$ hence we get the following corollary by the same argument. 

\begin{corollary}
The projective coordinate algebra $R_{C, \vec{p}}((1, 1), \ldots, (1, 1), 3)$ is Gorenstein with $*-$canonical
generator in degree $2.$
\end{corollary}

\section{The case $g = 0$}\label{g=0}

In this section we prove Theorem \ref{g0}. 
The semigroup $CB_{\Gamma}^*$ is highly dependent on the topology of the graph $\Gamma,$ however when the graph $\Gamma$ has genus $0,$ the situation simplifies considerably. From now on let $\tree$ denote a trivalent tree with $n$ leaves.  In both propositions below the active ingredient is that properties of $CB_3^*$ are preserved by fiber product over one of the three components $\partial_1, \partial_2, \partial_3$ of $\partial,$ because the values of these components always take the values $(0,1)$ ,$(1, 0),$ or $(0,0)$ on a point of $CB_3^*$.

\begin{proposition}
The polytope $CB_{\tree}^*$ is normal. 
\end{proposition}

\begin{proof}
This is a consequence of the argument used in \cite{MZ}, Theorem $1.2.$
\end{proof}

The restriction of any element  $\omega \in CB_{\tree}^*$ to a trinode is a member of $CB_3^*.$   This allows us to describe the lattice points in $CB_{\tree}^*$ as follows. First, we define an element of $CB_{\tree}^*$ to be $complete$ if its associated network on $\tree$ assigns every edge an arrow. 

\begin{figure}[htbp]
\centering
\includegraphics[scale = 0.65]{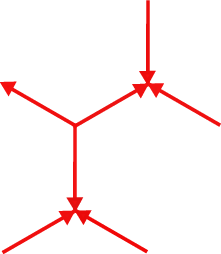}
\caption{A complete weighting on a tree.}
\label{fig:complete}
\end{figure} 

Note that there are exactly two complete weightings of any tree $\tree,$ determined
by the direction of the arrow on any leaf edge.   Recall that a disjoint union of trees is called a forest.  We say a forest $\mathcal{F} \subset \tree$ is $proper$ if it has at least two leaves, and every leaf of $\mathcal{F}$ is a leaf of $\tree.$  The following lemma is immediate.

\begin{figure}[htbp]
\centering
\includegraphics[scale = 0.65]{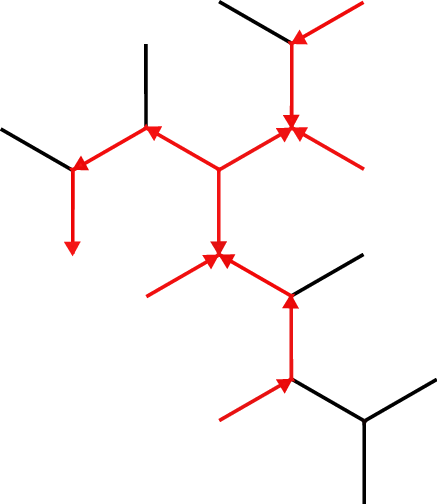}
\caption{A weighting on a subtree.}
\label{fig:proper}
\end{figure}

\begin{lemma}\label{forest}
Lattice points $\omega \in CB_{\tree}^*$ are in bijection with complete weightings on the components of proper forests in $\tree.$ 
\end{lemma}
 
The fact that any component of a proper forest in $\tree$ can be given two complete weightings, coupled with Lemma \ref{forest} give a complete combinatorial description of the generators of $CB_{\tree}^*,$   there are $3^{n-1}$ such weightings. Next we describe the binomial ideal which presents the affine semigroup algebra associated to $\C[CB_{\tree}^*].$ 

\begin{proposition}
The ideal $I_{\tree}$ which presents the toric algebra $\C[CB_{\tree}^*]$ by its degree $1$ component is generated by forms of degree $2$ and $3.$
\end{proposition}

\begin{proof}
This is also a consequence of the argument for Theorem $1.2$ in \cite{MZ}.
\end{proof}

Theorem \ref{g0} then follows by a flat degeneration argument.  In a reduced family of graded algebras, the generic maximum degrees of required generators and relations are bounded by the degrees at a closed point.  For an application of this argument over the stack $\bar{\mathcal{M}}_{g, n}$, see \cite{M7}.

\section{The case $g > 0$}\label{g=1}

In this section we show that the degree of generation of $V_{C, \vec{p}}(SL_3(\C))$ 
for $(C, \vec{p}) \in \bar{\mathcal{M}}_{g, n}$ is essentially controlled by the behavior of a specific polytope.   The wealth of polytopes $BZ^*_{\Gamma}$ at our disposal allows us the advantage of picking one with favorable properties.  We let $BZ^*_{ g, n}$ be the polytope obtained by fiber-product, dual to the graph $\Gamma(g, n)$ depicted in Figure \ref{fig:gn}.

\begin{figure}[htbp]
\centering
\includegraphics[scale = 0.55]{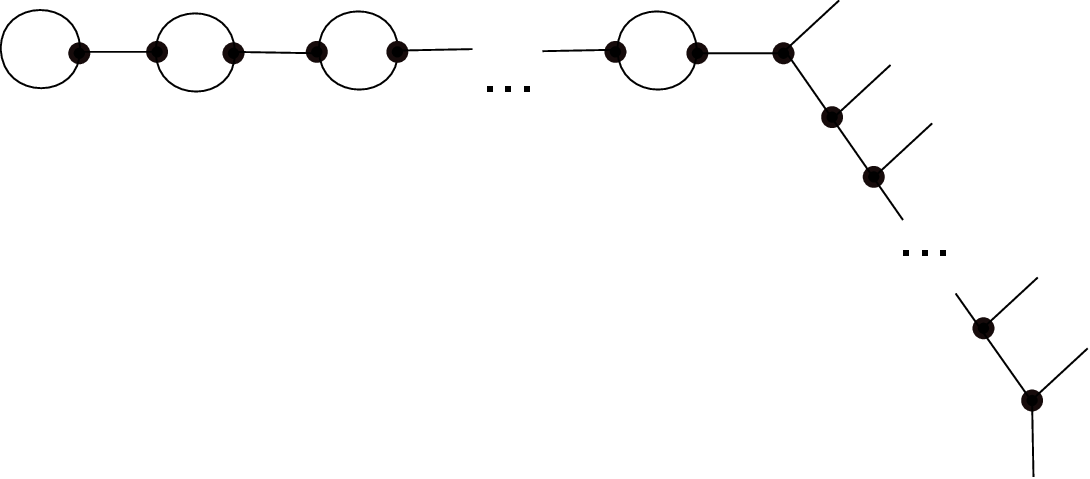}
\caption{The graph $\Gamma(g, n)$, which has $g$ loops and $n$ edges.}
\label{fig:gn}
\end{figure}

For $g \geq 1$, the polytope $BZ^*_{g, n}$ itself is a fiber product of of the polytope $BZ_{\tree_0}^*$ corresponding to the tree $\tree_0$ and the polytope $BZ^*_{g, 1}.$
As a result, we can control the behavior of $BZ^*_{g,n}$ to a certain degree. 

\begin{proposition}\label{bound}
The semigroup algebra $\C[BZ^*_{g, n}]$ is generated in degree bounded by the degree necessary to generate $\C[BZ_{g, 1}^*].$
\end{proposition}

\begin{proof}
We split the graph $\Gamma(g, n)$ at the edge $e$ which separates the tree $\tree_0$ component from the rest of the graph. 

\begin{equation}
\Gamma(g, n) = \Gamma(g, 1) \cup_e \tree_0\\
\end{equation}

\noindent
Consider an element $w \in L\circ BZ_{g, n}^*,$ and its restrictions $w|_{\tree_0}$ and $w|_{\Gamma(g, 1)}.$  We  once again use the argument from \cite{MZ}, Theorem $1.2$  to 
 show that any factorization of $w|_{\Gamma(g, 1)}$ can be extended to a factorization of all of $w.$ 
\end{proof}

Using this proposition, we can prove Theorem \ref{g1} with the following lemma.
The semigroup $BZ_{1, 1}^*$ is composed of those triangles which are dual to the graph composed of a single loop with an edge.  These are triangles such that the boundary values for two chosen edges are dual.  

\begin{lemma}
The semigroup algebra $\C[BZ^*_{1, 1}]$ is generated by elements of level $1, 2,$ and $3.$ \end{lemma}

\begin{proof}
We leave it to the reader to verify that the  elements depicted in Figure \ref{fig:11} suffice to generate $BZ_{1, 1}^*.$

\begin{figure}[htbp]
\centering
\includegraphics[scale = 0.35]{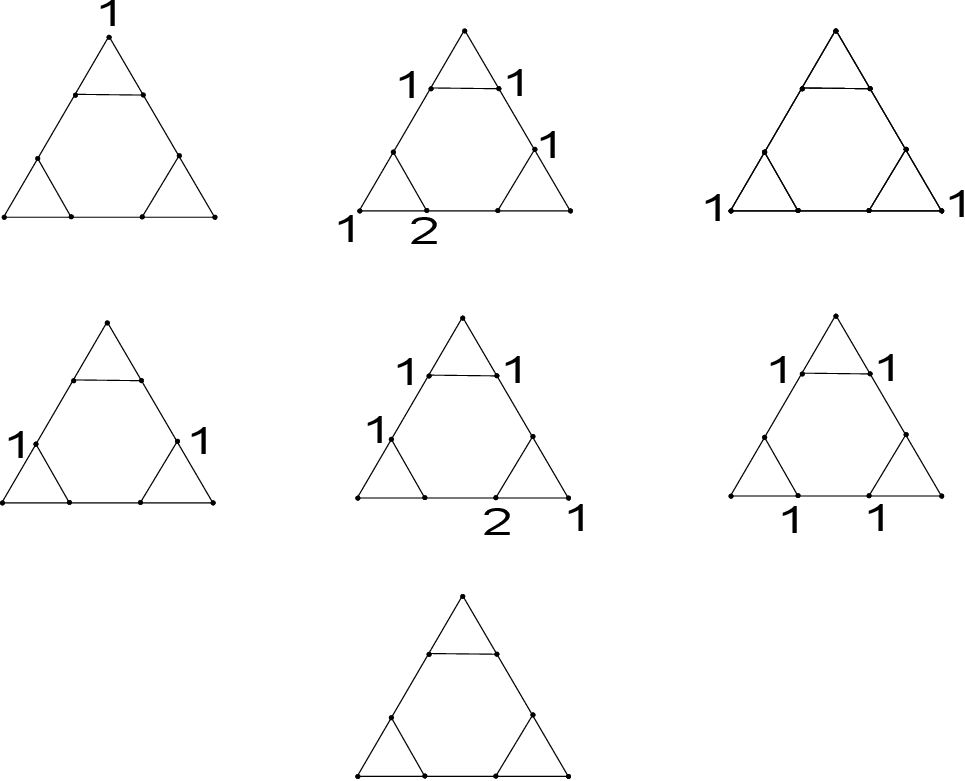}
\caption{Generators of $BZ_{1, 1}^*$.}
\label{fig:11}
\end{figure}

\end{proof}

\date{\today}

\end{document}